\documentclass[10pt]{amsart}
\usepackage{amsfonts, amsmath,amssymb,amsthm}
\usepackage[mathscr]{eucal}
\usepackage{hyperref}
\usepackage{verbatim}

\newtheorem{theorem}{Theorem}

\newtheorem{proposition}[theorem]{Proposition}

\theoremstyle{definition}
\newtheorem{definition}[theorem]{Definition}

\theoremstyle{remark}
\newtheorem{remark}[theorem]{Remark}

\newcommand{\set}[1]{\{ {#1} \}}

\newcommand{\wh}[1]{{\smash{\widehat{#1}}}}
\newcommand{\wt}[1]{{\smash{\widetilde{#1}}}}

\newcommand{\bbC}{\mathbb{C}}

\newcommand{\bbF}{\mathbb{F}}

\newcommand{\bbP}{\mathbb{P}}
\newcommand{\bbR}{\mathbb{R}}

\newcommand{\mbc}{\mathbf{c}}
\newcommand{\mbD}{\mathbf{D}}
\newcommand{\mbE}{\mathbf{E}}
\newcommand{\mbg}{\mathbf{g}}
\newcommand{\mbp}{\mathbf{p}}

\newcommand{\mcG}{\mathcal{G}}

\newcommand{\mcL}{\mathcal{L}}

\newcommand{\mfaut}{\mathfrak{aut}}
\newcommand{\mfg}  {\mathfrak{g}}

\newcommand{\mfh}  {\mathfrak{h}}
\newcommand{\mfhol}{\mathfrak{hol}}
\newcommand{\mfn}  {\mathfrak{n}}

\newcommand{\mfsl} {\mathfrak{sl}}
\newcommand{\mfso} {\mathfrak{so}}

\newcommand{\wtH}{\wt{H}}
\newcommand{\wtg}{\wt{g}}
\newcommand{\wtM}{\wt{M}}
\newcommand{\wtR}{\wt{R}}

\DeclareMathOperator{\End}{End}
\DeclareMathOperator{\G}{G}
\DeclareMathOperator{\GL}{GL}

\DeclareMathOperator{\Hol}{Hol}
\DeclareMathOperator{\id}{id}

\DeclareMathOperator{\Ooperator}{O}
\DeclareMathOperator{\Ric}{Ric}

\DeclareMathOperator{\SO}{SO}

\begin{document}

\title{Cartan's incomplete classification \\ and an explicit ambient metric of holonomy $\G_2^*$}

\dedicatory{Dedicated to C. Robin Graham on the occasion of his 60th birthday.}

\author{Travis Willse}

\address{
Fakult\"at f\"ur Mathematik \\
Universit\"at Wien \\
Oskar-Morgenstern-Platz 1, 1090 Wien \\
Austria}

\email{travis.willse@univie.ac.at}

\maketitle

\begin{abstract}
In his 1910 ``Five Variables'' paper, Cartan solved the equivalence problem for the geometry of $(2, 3, 5)$ distributions and in doing so demonstrated an intimate link between this geometry and the exceptional simple Lie groups of type $\G_2$. He claimed to produce a local classification of all such (complex) distributions which have infinitesimal symmetry algebra of dimension at least $6$ (and which satisfy a natural uniformity condition), but in 2013 Doubrov and Govorov showed that this classification misses a particular distribution $\mbE$. We produce a closed form for the Fefferman-Graham ambient metric $\smash{\wtg_{\mbE}}$ of the conformal class induced by (a real form of) $\mbE$, expanding the small catalogue of known explicit, closed-form ambient metrics. We show that the holonomy group of $\smash{\wtg_{\mbE}}$ is the exceptional group $\smash{\G_2^*}$ and use that metric to give explicitly a projective structure with normal projective holonomy equal to that group.

We also present some simple but apparently novel observations about ambient metrics of general left-invariant conformal structures that were used in the determination of the explicit formula for $\smash{\wtg_{\mbE}}$.
\end{abstract}


\markboth{}{}

\thispagestyle{empty}

\tableofcontents

\section{Introduction}\label{section:introduction}
Cartan's most involved application of his powerful method of equivalence was the resolution of the (local) equivalence problem for the geometry of $2$-plane distributions $\mbD$ on $5$-manifolds $M$, which he reported in his landmark 1910 ``Five Variables'' paper \cite{Cartan}. Such distributions that are maximally nonintegrable in the sense that
\[
    [\mbD, [\mbD, \mbD]] = TM
\]
are called \textbf{$(2, 3, 5)$ distributions} because $2, 3, 5$ are the respective ranks of the vector bundles in the derived filtration $\mbD \subset [\mbD, \mbD] \subset TM$ of $TM$. This geometry has proved interesting for numerous reasons: It comprises a first class of distributions with continuous local invariants, it is connected to the geometry of surfaces rolling on one another without slipping or twisting \cite{AnNurowski,BorMontgomery,BryantHsu} and hence to natural problems in control theory \cite{AgrachevSachov}, it is intimately linked (as Cartan showed in the aforementioned article\footnote{In fact, Cartan \cite{CartanStructureOfGroups} and Engel \cite{Engel} (separately) first discovered this relationship in simultaneous notes in 1893: They gave explicit distributions on $\bbC^5$ whose infinitesimal symmetry algebras are isomorphic to the complex simple Lie algebra $\mfg_2^{\bbC}$, hence furnishing the first explicit realization of an exceptional such algebra. See \cite{Agricola} for a concise history of the topic.}) to the exceptional simple complex Lie group of type $\G_2^{\bbC}$ and then via the natural inclusion $\G_2^* \hookrightarrow \SO(3, 4)$ to conformal geometry \cite{NurowskiDifferential}, and it is the appropriate structure for naturally projectively compactifying strictly nearly \mbox{(para-)}K\"{a}hler structures in dimension $6$ \cite{GoverPanaiWillse}. (Here, $\G_2^*$ denotes the automorphism group of the algebra of split octonions, one of the split real forms of the complex. connected simple Lie group $\G_2^{\bbC}$.) In turn, via the Fefferman-Graham ambient construction \cite{FeffermanGraham} and a general extension theorem \cite{GrahamWillse} these structures can be used to construct metrics of holonomy contained in and sometimes equal to $\G_2^*$.

Cartan claimed to classify locally all (implicitly, complex) $(2, 3, 5)$ distributions $(M, \mbD)$ that both (1) have infinitesimal symmetry algebra $\mfaut(\mbD)$ of dimension at least $6$ and (2) satisfy a natural uniformity condition called \textit{constant root type} (see \S\ref{section:missing-distribution}). The algebra $\mfaut(\mbD)$ consists of all vector fields $X \in \Gamma(TM)$ whose flow preserves $\mbD$, or equivalently, for which $\mcL_X Y \in \Gamma(\mbD)$ for all $Y \in \Gamma(\mbD)$. Some 103 years later, Doubrov and Govorov upended this classification \cite{DoubrovGovorovCounterexample} by showing that it misses a distribution, which we denote $\mbE$, and which turns out to be a left-invariant distribution on a particular Lie group. In fact, Doubrov and Govorov reported finding this missing distribution to the author during his preparation of an earlier article \cite{WillseHeisenberg} that relied on Cartan's classification; this distribution was discussed briefly in Example 38 and Remark 39 in that article, but detailed discussion was deferred to the present article.

Nurowski showed that any $(2, 3, 5)$ distribution $(M, \mbD)$ canonically induces a conformal structure $\mbc_{\mbD}$ of signature $(2, 3)$ on $M$ \cite{NurowskiDifferential}, and he later showed with Leistner that one can exploit certain distributions of this type to produce explicit \textit{ambient metrics} $\wtg$ with metric holonomy $\Hol(\wtg)$ equal to $\G_2^*$. The ambient metric construction assigns to a conformal structure $(M, \mbc)$ of signature $(p, q)$ an essentially unique Riemannian metric $(\wtM, \wtg)$ of signature $(p + 1, q + 1)$; see \S\ref{subsection:ambient-metric-construction}. Though the ambient metric construction enjoys a satisfying existence theorem, producing an explicit ambient metric for a particular conformal structure amounts to solving a formidable nonlinear system of partial differential equations on $M$, and consequently explicit closed forms for ambient metrics are known only for a few special classes of conformal structures \cite{AndersonLeistnerNurowski}, \cite[\S3]{FeffermanGraham}, \cite[Theorem 2.1]{GoverLeitner}, \cite{LeistnerNurowskiPPwave}.

We observe in \S\ref{subsection:ambient-left-invariant} that in the special case that $\mbc$ is a left-invariant conformal structure on a Lie group, after making certain natural choices the system becomes one of \textit{ordinary} differential equations. In particular this applies to the left-invariant conformal structure $\mbc_{\mbE}$ determined by $\mbE$; here we abuse notation by using the same notation $\mbE$ for a particular real distribution, as well as Doubrov and Govorov's missing complex distribution, which is in a sense that can be made precise a complexification of that real distribution. In \S\ref{section:new-explicit-ambient-metric} we manage to solve explicitly the system of ordinary differential equations determined by $\mbc_{\mbE}$ and hence produce for it an explicit ambient metric $\wtg_{\mbE}$, and its behavior turns out to be qualitatively different from previous explicit examples; see Remark \ref{remark:nonpolynomial}. Furthermore, computing using the Ambrose-Singer Theorem gives that the holonomy group of $\wtg_{\mbE}$ is $\G_2^*$, furnishing a new example a metric with that exceptional holonomy group. Finally, we apply some general facts relating ambient metrics to projective structures whose normal tractor connections enjoy orthogonal holonomy reductions \cite{GoverPanaiWillse} to produce from $\wtg_{\mbE}$ an explicit example of a projective structure with normal projective holonomy $\G_2^*$.


Many computations whose results are reported here were carried out with Ian Anderson's Maple Package \texttt{DifferentialGeometry}.

\thanks{It is a pleasure to thank Mike Eastwood, who encouraged the author to describe this example in a standalone article and offered helpful comments during its preparation. It is also a pleasure to thank Boris Doubrov for several illuminating discussions, and to thank Matthew Randall for several corrections. I am grateful, too, to referees for several helpful suggestions and another correction. The author gratefully acknowledges support from the Australian Research Council and the Austrian Science Fund (FWF), the latter via project P27072--N25.}

\section{$(2, 3, 5)$ distributions}\label{section:235-distributions}

Given two distributions $\mbD, \mbD' \subset TM$ on a (real or complex) manifold, their bracket $[\mbD, \mbD']$ is the set $[\mbD, \mbD'] := \coprod_{u \in M} \set{[X, Y]_u : X \in \Gamma(\mbD), Y \in \Gamma(\mbD')}$. (A priori this set need not be a distribution, as the rank of the vector space $[\mbD, \mbD'] \cap T_u M$ may vary with $u$.) Short computations give that for a distribution $\mbD$ we have $\mbD \subseteq [\mbD, \mbD]$, and that if $[\mbD, \mbD]$ is a distribution, then $[\mbD, \mbD] \subseteq [\mbD, [\mbD, \mbD]]$. (Implicit in the notation $[\mbD, [\mbD, \mbD]]$ is the requirement that $[\mbD, \mbD]$ itself be a distribution.)

For reasons including those mentioned in the introduction, the class of distributions that the current article concerns is especially interesting:

\begin{definition}
A \textbf{$(2, 3, 5)$ distribution}\footnote{Elsewhere $(2, 3, 5)$ distributions are sometimes instead termed \textit{generic} $2$-plane distributions (on $5$-manifolds) \cite{CapSlovak,Sagerschnig}.} is a $2$-plane distribution $\mbD$ on a (real or complex) $5$-manifold $M$ that satisfies the genericity condition
\[
	[\mbD, [\mbD, \mbD]] = TM \textrm{.}
\]
Two such distributions $(M, \mbD)$ and $(M', \mbD')$ are \textbf{equivalent} iff there is a diffeomorphism $\phi: M \to M'$ such that $T\phi \cdot \mbD = \mbD'$. They are \textbf{locally equivalent (near $u \in M$ and $u' \in M'$)} if there are neighborhoods $U$ of $u$ and $U'$ of $u'$ such that $(U, \mbD|_U)$ and $(U', \mbD'|_{U'})$ are equivalent via a diffeomorphism that maps $u$ to $u'$.
\end{definition}

The geometry of these structures was first studied systematically by Cartan, in his well-known ``Five Variables'' article \cite{Cartan}. There, he solved the equivalence problem for this geometry by canonically assigning to each such distribution $(M, \mbD)$ a principal $Q$-bundle $E \to M$ and a $\mathfrak{g}_2^*$-valued pseudoconnection $\omega$ on $E$, where $Q$ is a particular parabolic subgroup of $\G_2^*$.\footnote{Since $\G_2^*$ is semisimple and $Q$ is parabolic, this geometry is an example of an important class of structures called \textit{parabolic geometries}. For any such geometry one can encode any structure on a manifold $M$ in a bundle $E \to M$ equipped with a canonically determined Cartan connection; see the standard reference \cite{CapSlovak} for a general treatment of parabolic geometries and \S4.3.2 there for discussion of the geometry of $(2, 3, 5)$ distributions in this setting. Robert Bryant has observed that the pseudoconnection Cartan produces is not in fact a Cartan connection \cite{BryantPrivate}.}

We here consider $(2, 3, 5)$ distributions whose \textit{infinitesimal symmetry algebra} is large:

\begin{definition}
A vector field $X \in \Gamma(TM)$ is an \textbf{infinitesimal symmetry} of a distribution $(M, \mbD)$ iff it preserves $\mbD$ in the sense that $\mcL_X Y \in \Gamma(\mbD)$ for all $Y \in \Gamma(\mbD)$. The Lie algebra $\mfaut(\mbD)$ of all such fields is the \textbf{infinitesimal symmetry algebra} of $(M, \mbD)$.
\end{definition}

\subsection{Monge (quasi)normal form}

This geometry admits a local quasinormal form: Any ordinary differential equation
\begin{equation}\label{equation:monge-normal-form}
    z'(x) = F(x, y(x), y'(x), y''(x), z(x)),
\end{equation}
can be regarded as a differential system on the partial jet space $\bbF^5_{xypqz} := J^{2, 0}(\bbF, \bbF^2)$ (here $\bbF = \bbR$ or $\bbF = \bbC$), where the variables $p$ and $q$ are placeholders respectively for $y'$ and $y''$, comprising the canonical forms $dy - p \,dx$ and $dp - q \,dx$, and the particular form $dz - F(x, y, p, q, z) \,dx$\textrm{.}
These forms are linearly independent at each point, so their common kernel $\mbD_F$ is a $2$-plane distribution on $\bbF^5_{xypqz}$, namely,
\begin{equation}
    \mbD_F = \langle \partial_x + p \partial_y + q \partial_p + F(x, y, p, q, z) \partial_z, \partial_q \rangle \textrm{.}
\end{equation}
Concretely, a pair $(y(x), z(x))$ is a solution of the equation \eqref{equation:monge-normal-form} iff its prolongation $x \mapsto (x, y(x), y'(x), y''(x), z(x))$
is an integral curve of $\mbD_F$. Computing directly shows that $\mbD_F$ is a $(2, 3, 5)$ distribution iff the second derivative $F_{qq}$ vanishes nowhere.

Goursat showed that every $(2, 3, 5)$ distribution is locally equivalent (around any point) to $\mbD_F$ for some $F$ \cite[\S76]{Goursat}. A distribution $\mbD_F$ specified by a function $F$ is said to be in \textbf{Monge (quasi)normal form}.

\subsection{The canonical conformal structure}\label{subsection:canonical-conformal-strucuture}

Nurowski showed that any real $(2, 3, 5)$ distribution $(M, \mbD)$ determines a canonical conformal structure $\mbc_{\mbD}$ on the underlying manifold $M$ \cite[\S5.3]{NurowskiDifferential}, and the argument in that reference shows that Nurowski's construction applies just as well to establish that a complex $(2, 3, 5)$ distribution induces a complex conformal structure on the underlying manifold.

Nurowski also gave an explicit (and, with a length of about 60 terms, daunting) formula \cite[(54)]{NurowskiDifferential} for a representative of the conformal structure $\mbc_{\mbD_F}$ of a distribution $\mbD_F$ in Monge normal form; its co\"{e}fficients in the co\"{o}rdinate coframe are polynomials of degree $6$ and lower in the $4$-jet of $F$.

\section{A missing distribution}
\label{section:missing-distribution}

\subsection{Cartan's ostensible classification}
\label{subsection:ostensible-classification}
In \cite{Cartan} Cartan claimed to classify, up to local equivalence, and implicitly in the complex setting, all $(2, 3, 5)$ distributions whose infinitesimal symmetry algebra has dimension at least $6$ and which have constant \textit{root type}: The fundamental curvature invariant of a $(2, 3, 5)$ distribution $(M, \mbD)$---analogous to the Riemann curvature tensor in Riemannian geometry---is a section $A \in \Gamma(\bigodot^4 \mbD^*)$. The value of $A$ at $u \in M$ depends on the $4$-jet of $\mbD$ at $u$, and the quantity $A$ is natural in the sense that it is preserved by diffeomorphism. The \textbf{root type} of $\mbD$ at $x \in M$ is just the collection of multiplicities of the zeros of $A_x$ viewed as a polynomial on the projective line $\bbP(\mbD)$, or if $\mbD$ is real, on $\bbP(\mbD \otimes \bbC)$.\footnote{Cf.\ the notion of Petrov type in four-dimensional Lorentzian geometry.} A distribution has \textbf{constant root type} if the root type is the same for every point, and of course, homogeneous $(2, 3, 5)$ distributions have this property.

The structures he found were the following, organized by root type. (The section numbers indicate the corresponding (sub)sections in \cite{Cartan}.)
\begin{itemize}
	\item (\S8: $A$ identically zero) The \textit{flat model} $(\G_2^* / Q, \Delta)$ of the geometry, which can be realized locally (1) via the algebra of the split octonions \cite{Sagerschnig}, (2) as the so-called \textit{rolling distribution} for two spheres whose radii have ratio $3 : 1$ (see the item for $\S50$ below for references), and (3) as the distribution given in Monge normal form $F(x, y, p, q, z) = q^2$ \cite{NurowskiDifferential}; $\mfaut(\Delta) = \mfg_2^*$.
    \item (\S9: $A$ has a single quadruple root at each point) A class of distributions $\mbD_I$ specified by a single function $I$ of one variable. If $I$ is constant then $\mbD_I$ is homogeneous and $\dim \mfaut(\mbD_I) = 7$; otherwise $\dim \mfaut(\mbD_I) = 6$ and $\mbD_I$ is not homogeneous. In both cases $\mfaut(\mbD_I)$ is solvable.
   	\item (\S11: $A$ has two double roots at each point) A family of distributions parameterized by one complex parameter. Up to isomorphism, three different infinitesimal symmetry algebras occur, all of which have dimension $6$.
   		\begin{itemize}
   			\item (\S50) A general distribution in the family has infinitesimal symmetry algebra isomorphic to $\mfso(3, \bbC) \oplus \mfso(3, \bbC)$. This family includes appropriate complexifications of the rolling distributions of two real surfaces of different nonzero constant curvature whose curvatures do not have ratio $9 : 1$; for that ratio the rolling distribution is flat. See \cite{Agrachev,BaezHuerta,BorMontgomery,BryantHsu,Zelenko} for much more.
   			\item (\S51) There is a single distribution with infinitesimal symmetry algebra isomorphic to $\mfso(3, \bbC) \oplus (\mfso(2, \bbC) \ltimes \bbC^2)$.
   			\item (\S52) There is a single distribution with infinitesimal symmetry algebra isomorphic to $\mfso(3, \bbC) \ltimes \bbC^3$. This distribution is locally equivalent to the appropriate complexification of the rolling distribution of a real surface of constant curvature and a plane.
   		\end{itemize}
\end{itemize}

\subsection{Doubrov--Govorov's distribution}
\label{subsection:doubrov-example}

Some 103 years after Cartan's work, Doubrov and Govorov found a complex $(2, 3, 5)$ distribution $\mbE$ that has $6$-dimensional infinitesimal symmetry algebra but which is missing from Cartan's classification.

Define a Lie algebra structure $\mfh$ on $\bbF^5$ ($\bbF = \bbR$ or $\bbF = \bbC$) by the bracket multiplication table
\begin{equation}\label{equation:bracket-relations}
	\begin{array}{c|ccccc}
		[\,\cdot\, , \,\cdot\,]    & E_1 & E_2 & E_3 & E_4 & E_5 \\
	\hline
		E_1 & \cdot & \cdot & \cdot & \cdot & \cdot \\
		E_2 & & \cdot & \cdot & E_1 & -E_2 \\
		E_3 & & & \cdot & E_2 & -2 E_3 \\
		E_4 & & & & \cdot & E_4 \\
		E_5 & & & & & \cdot
	\end{array} \quad ,
\end{equation}
where $(E_a)$ is any basis, and define $\mbE_0 := \langle E_1 + E_3 + E_4, E_5 \rangle$\textrm{.}
Then, let $H$ be a connected Lie group with Lie algebra $\mfh$, regard $\mbE_0$ as a subspace of $T_{\id} H$ via its identification with $\mfh$, and let $\mbE \subset TH$ be the left-invariant $2$-plane distribution on $H$ characterized by $\mbE_{\id} = \mbE_0$. Via a usual abuse of notation, let $(E_a)$ also denote the left-invariant frame on $H$ whose restriction to $\id \in H$ is the basis $(E_a) \subset \mfh \cong T_{\id} H$; then, computing using the above bracket relations gives $[\mbE, \mbE] = \langle E_1 + E_3, E_4, E_5 \rangle$ and $[\mbE, [\mbE, \mbE]] = TH ,$ and so $\mbE$ is a $(2, 3, 5)$ distribution.

Doubrov and Govorov integrated these structure equations and produced a Monge normal form for the distribution, given by the function $F_{\mbE}(x, y, p, q, z) := y + q^{1 / 3}$, on a suitable domain. For readability, we use co\"{o}rdinates $(x, y, p, r, z)$, where $q = r^3$.

One can verify directly that $F_{\mbE}$ defines a Monge normal form for $\mbE$ by showing that the frame $L_a$ defined by
\[
    L_1 :=     \partial_z , \quad
    L_2 := r   \partial_y , \quad
    L_3 := r^2 \partial_p , \quad
    L_4 := \tfrac{1}{r} (\partial_x + p \partial_y + y \partial_z) , \quad
    L_5 := r   \partial_r
\]
satisfies the bracket relations of $(E_a)$ and that the identifications $E_a \leftrightarrow L_a$ identify $\mbE$ with the distribution determined by $F_{\mbE}$ (we henceforth make these identifications locally).

Since $\mbE$ is left-invariant, its infinitesimal symmetry algebra $\mfaut(\mbE)$ contains the Lie algebra of right-invariant vector fields on $H$; one basis of this subalgebra is
\[
	\left(
		\partial_z ,
	    \partial_y + x \partial_z ,
	    x \partial_y + \partial_p + \tfrac{1}{2} x^2 \partial_z ,
	    \partial_x ,
	    -x \partial_x + y \partial_y + 2 p \partial_p + r \partial_r
	\right) \textrm{.}
\]
Direct computation shows that the vector field $-y \partial_x + p^2 \partial_p + p r \partial_r - \tfrac{1}{2} y^2 \partial_z$, which is not right-invariant, is also in $\mfaut(\mbE)$. Tedious verification shows that this field and the right-invariant fields together span the infinitesimal symmetry algebra and establishes that $\mfaut(\mbE) \cong \mfsl(2, \bbF) \ltimes \mfn_3$, where $\mfn_3$ is the $3$-dimensional Heisenberg algebra over $\bbF$, and hence $\dim \mfaut(\mbE) = 6$.

\begin{proposition}\cite{DoubrovGovorovCounterexample}\label{proposition:counterexample}
The (complex) homogeneous $(2, 3, 5)$ distribution $(H, \mbE)$ is not locally equivalent to any of the distributions in Cartan's list. In particular, Cartan's list is incomplete.
\end{proposition}

\begin{proof}
Computing gives that the fundamental curvature of $\mbE$ is $A_{\mbE} = \frac{1}{4} (e^4)^4 \vert_{\mbE}$ (here, $(e^a)$ is the left-invariant coframe on $H$ dual to $(E_a)$), so at each point it has a quadruple root at $[E_5] \in \bbP(\mbE)$. But there is no homogeneous distribution $\mbD$ in Cartan's list with this (constant) root type for which $\dim\mfaut(\mbD) = 6$, so $\mbE$ is not equivalent to any distribution on that list.
\end{proof}

\begin{remark}
One can establish Proposition \ref{proposition:counterexample} without the work of verifying that $\dim\mfaut(\mbE)$ is not larger than $6$ by observing that all of the distributions $\mbD$ in Cartan's list for which $A_{\mbD}$ has a quadruple root at every point have solvable infinitesimal symmetry algebras, whereas $\mfaut(\mbE)$ contains the simple subalgebra $\mfsl(2, \bbC)$ and so is not solvable.
\end{remark}

\section{Ambient metrics}\label{section:ambient-metrics}

The Fefferman-Graham ambient construction associates to a conformal structure $(M, \mbc)$ of signature $(p, q)$, $\dim M = p + q \geq 2$, a unique pseudo-Riemannian metric $(\wtM, \wtg)$ of signature $(p + 1, q + 1)$. When $\dim M$ is odd, this construction is essentially unique, and hence invariants of $\wtg$ (that are independent of the choices made) are also invariants of the underlying conformal structure; indeed, this was the original motivation for the construction \cite{FeffermanGraham}.

\subsection{The Fefferman-Graham ambient metric construction}\label{subsection:ambient-metric-construction}

We describe the ambient metric construction for odd dimension $n > 1$ following \cite{FeffermanGraham}.

Fix a conformal structure $\mbc$ of dimension $n$. The \textbf{metric bundle} associated to $\mbc$ is the ray bundle $\pi: \mcG \to M$ defined by
\[
	\mcG := \coprod_{u \in M} \set{g_u : g \in \mbc}.
\]
The bundle $\mcG$ enjoys a natural topology and smooth structure such that its smooth sections are precisely the representative metrics of $\mbc$. The natural dilation action $\bbR_+ \times \mcG \to \mcG$ defined by $s \cdot g_u = \delta_s(g_u) := s^2 g_u$ realizes $\mcG$ as a principal $\bbR_+$-bundle. Also, $\mcG$ admits a tautological symmetric $2$-tensor $\smash{\mbg_0 \in \Gamma(\bigodot^2 T^* \mcG)}$ defined by
\[
    (\mbg_0)_{g_u} (\xi, \eta) = g_u(T_{g_u} \pi \cdot \xi, T_{g_u} \pi \cdot \eta) \textrm{;}
\]
$\mbg_0$ annihilates $\pi$-vertical vectors and hence is degenerate. By construction, $\delta_s^* \mathbf{g}_0 = s^2 \mathbf{g}_0$.

Fixing a representative $g \in \mbc$ trivializes $\mcG \leftrightarrow \bbR_+ \times M$ via the identification $t^2 g_u \leftrightarrow (t, u)$. With respect to this trivialization, the tautological $2$-tensor is given by $\mbg_0 = t^2 \pi^* g$ and the dilations by $\delta_s: (t, u) \mapsto (st, u)$.

Now, consider the space $\mcG \times \bbR$ and denote the standard co\"{o}rdinate on $\bbR$ by $\rho$. Then, the map $\iota: \mcG \hookrightarrow \mcG \times \bbR$ defined by $z \mapsto (z, 0)$ embeds $\mcG$ as a hypersurface in $\mcG \times \bbR$ and we identify $\mcG$ with this hypersurface. The dilations $\delta_s$ extend trivially to $\mcG \times \bbR$, that is, by $s \cdot (g_u, \rho) = \delta_s (g_u, \rho) := (s^2 g_u, \rho)$. Likewise, a choice of representative $g \in \mbc$ determines a trivialization $\mcG \times \bbR \leftrightarrow \bbR_+ \times M \times \bbR$ that identifies $(t^2 g_u, \rho) \leftrightarrow (t, u, \rho)$, which in turn defines an embedding $M \hookrightarrow \mcG \times \bbR$ by $u \mapsto (1, u, 0)$ and yields an identification $T(\mcG \times \bbR) \cong \bbR \oplus TM \oplus \bbR$. As is conventional, we denote indices corresponding to the factor $\bbR_+$ by $0$, those corresponding to $M$ by lowercase Latin letters, $a, b, c, \ldots$, and those corresponding to the factor $\bbR$ by $\infty$.

A smooth metric $\wtg$ of signature $(p + 1, q + 1)$ on an open neighborhood $\wtM$ of $\mcG$ in $\mcG \times \bbR$ invariant under the dilations $\delta_s$, $s \in \bbR_+$, is a \textbf{pre-ambient metric} for $(M, \mbc)$ if
\begin{enumerate}
    \item it extends $\mbg_0$ in the sense that $\iota^* \wtg = \mbg_0$, and
    \item it is homogeneous of degree $2$ with respect to the dilations $\delta_s$, that is, $\delta_s^* \wtg = s^2 \wtg$.
\end{enumerate}
A pre-ambient metric is \textbf{straight} if for all $u \in \wtM$ the parameterized curve $\bbR_+ \to \wtM$ defined by $s \mapsto s \cdot u$ is a geodesic. Any (nonempty) conformal structure admits many pre-ambient metrics; Cartan's normalization condition for a conformal connection \cite{CartanConformal} suggests that Ricci-flatness is a natural distinguishing criterion.

\begin{definition}\label{definition:ambient-metric}
Let $(M, \mbc)$ be a conformal manifold of odd dimension at least $3$. An \textbf{ambient metric} for $(M, \mbc)$ is a straight pre-ambient metric $\wtg$ for $(M, \mbc)$ such that the Ricci curvature $\wtR$ of $\wtg$ is $O(\rho^{\infty})$; the pair $(\wtM, \wtg)$ is an \textbf{ambient manifold} for $(M, \mbc)$.
\end{definition}
Here, we say that a tensor field on $\wtM$ is $O(\rho^{\infty})$ if it vanishes to infinite order in $\rho$ at each point in the zero set $\mcG$ of $\rho$.

We formulate Fefferman--Graham's existence and uniqueness results for ambient metrics of odd-dimensional conformal structures as follows:
\begin{theorem}\label{theorem:Fefferman-Graham}\cite{FeffermanGraham}
Let $(M, \mbc)$ be a conformal manifold of odd dimension at least $3$. There is an ambient metric for $(M, \mbc)$, and it is unique up to infinite order: If $\wtg_1$ and $\wtg_2$ are ambient metrics for $(M, \mbc)$, then (after possibly restricting the domains of both to appropriate open neighborhoods of $\mcG$ in $\wtM$) there is a diffeomorphism $\Phi$ such that $\Phi\vert_{\mcG} = \id_{\mcG}$ and $\Phi^* \wtg_2 - \wtg_1$ is $O(\rho^{\infty})$.
\end{theorem}

The proof of the theorem uses the fact that for any representative $g \in \mbc$ we may write any ambient metric $\wtg$ (possibly after restricting to an open, dilation-invariant set containing $\mcG$) in the \textit{normal form}
\begin{equation}\label{equation:normal-form}
    \wtg = 2 \rho \, dt^2 + 2 t \, dt \, d\rho + t^2 g(u, \rho)
\end{equation}
where $g(u, 0) = g$ and where we may regard $g(u, \rho)$ as a family of metrics on $M$ depending on the parameter $\rho$; here and henceforth we suppress notation for pullback by the inclusion $M \hookrightarrow \wtM$ determined by $g$.

We want to write the Ricci-flatness condition $\wtR = 0$ for a metric $\wtg$ on $\wtM$ in normal form: The components $\wtR_{00}, \wtR_{0a}, \wtR_{0\infty}$ of $\wt{\Ric} := \Ric^{\wtg}$ with respect to the splitting $\bbR_+ \times M \times \bbR$ of $\mcG \times \bbR \supseteq \wtM$ are zero. The vanishing of the remaining components $\wtR_{ab}, \wtR_{a\infty}, \wtR_{\infty\infty}$ is equivalent to a system of $\frac{1}{2}(n + 2)(n + 1)$ partial differential equations in $g_{ab}$; these components are given by \cite[(3.17)]{FeffermanGraham}
\begin{align}\label{equation:ambient-metric-PDE}
    \wtR_{a     b     } &= \rho g_{ab}'' - \rho g^{cd} g_{ac}' g_{bd}' + \tfrac{1}{2} \rho g^{cd} g_{cd}' g_{ab}' - \left( \tfrac{n}{2} - 1 \right) g_{ab}' - \tfrac{1}{2} g^{cd} g_{cd}' g_{ab} + R_{ab} \nonumber \\
    \wtR_{a     \infty} &= \tfrac{1}{2} g^{cd} (\nabla_c g_{ad}' - \nabla_a g_{cd}') \\
    \wtR_{\infty\infty} &= -\tfrac{1}{2} g^{cd} g_{cd}'' + \tfrac{1}{4} g^{cd} g^{ef} g_{ce}' g_{df}' \nonumber
\textrm{;}
\end{align}
here $g = g(x, \rho)$, $'$ denotes differentiation with respect to $\rho$, and $\nabla_a$ and $R_{ab}$ respectively denote the Levi-Civita connection and Ricci curvature of $g(x, \rho)$ with $\rho$ fixed.

Differentiating these expressions, setting them equal to zero, and then evaluating at $\rho = 0$ successively determines all of the derivatives $\partial_{\rho}^m g\vert_{\rho = 0}$. 
If the representative $g$ is real-analytic, then the Taylor series of $g(u, \rho)$ about $\rho = 0$ (that is, along $\mcG$) converges to a real-analytic ambient metric $\wtg$ on some open subset of $\wtM$ containing $\mcG$, and in particular $\wt\Ric = 0$.

Despite the existence of ambient metrics that Theorem \ref{theorem:Fefferman-Graham} guarantees, explicit examples have been produced for only a few isolated classes of conformal structures \cite[\S3]{FeffermanGraham}, \cite[Theorem 2.1]{GoverLeitner}, \cite{LeistnerNurowskiPPwave}, \cite[\S2]{HSSTZ}, owing in part to the severe nonlinearity of the system $\wt\Ric = 0$. Nurowski produced explicit ambient metrics for the conformal structures induced by a special class of $(2, 3, 5)$ distributions \cite{NurowskiAmbient} and with Leistner showed that most of these metrics have holonomy $\G_2^*$ \cite{LeistnerNurowski}; the same has been accomplished for a proper superset of this family, namely for the conformal structures induced by the distributions given in Monge normal form by the functions $F_{f, h}(x, y, p, q, z) := q^2 + f(x, y, p) + h(x, y) z$ \cite{AndersonLeistnerNurowski}.

\subsection{Ambient metrics of left-invariant conformal structures}\label{subsection:ambient-left-invariant}

The typically intractable problem of computing explicit ambient metrics of particular conformal structures in principle simplifies significantly (but in general remains difficult) in the special case of left-invariant conformal structures on Lie groups. We indicate some of these simplifications here and treat this problem in more detail in a work in progress \cite{WillseLeftInvariant}.

Given a Lie group $G$, let $L_h : G \to G$ be the left multiplication map $k \mapsto hk$. A conformal structure $\mbc$ on a Lie group $G$ is \textbf{left-invariant} iff $L_h^* \mbc = \mbc$ for all $h \in G$, that is, iff for any (equivalently, every) representative metric $g \in \mbc$ we have $L_h^* g \in \mbc$. Any such conformal structure contains a distinguished $1$-parameter family of representative metrics, namely the left-invariant ones: Any $g_0 \in \mbc\vert_{\id}$ determines a unique left-invariant metric $g$ satisfying $g\vert_{\id} = g_0$, and by left-invariance $g\vert_h \in \mbc\vert_h$ for all $h \in G$, that is, $g \in \mbc$. Since the choice of $g_0 \in \mbc\vert_{\id}$ is arbitrary, all left-invariant metrics in $\mbc$ arise this way.

The problem of finding an explicit expression for the ambient metric simplifies in a critical way when writing the ambient metric in normal form \eqref{equation:normal-form} with respect to a left-invariant representative metric $g$.

\begin{proposition}\label{proposition:left-invariant-ambient-metric}
Let $(G, \mbc)$ be a left-invariant conformal structure on a Lie group of odd dimension $n \geq 3$ and $g \in \mbc$ a left-invariant representative metric, and let $(E_a)$ be any left-invariant frame on $G$.

\begin{enumerate}
	\item There is a real-analytic---and hence Ricci-flat---ambient metric $\wtg$ for $\mbc$ invariant under the trivial extension of the left action of $G$ to the ambient space $\wt G \subseteq \bbR_+ \times G \times \bbR$. In particular, for fixed $\rho$ the quantity $g(h, \rho)$ in the normal form \eqref{equation:normal-form} is a left-invariant metric on $G$, and so the components of $g(h, \rho)$ with respect to the left-invariant frame $(E_a)$ are functions $g_{ab}(\rho)$ of $\rho$ alone.
	\item The components of the ambient metric system $\wt\Ric = 0$ with respect to the ($G$-invariant) frame $(\partial_t, E_a, \partial_{\rho})$ of $T\wt G$ comprise a system of ordinary differential equations in $g_{ab}(\rho)$.
\end{enumerate}
\end{proposition}

Theorem \ref{theorem:Fefferman-Graham} thus guarantees that any ambient metric for a left-invariant conformal structure is, informally, ``invariant to infinite order'' under the $G$-action on $\mcG$ specified therein.

\mbox{}
\begin{proof}~
\begin{enumerate}
	\item Since $g$ is left-invariant it is real-analytic, and hence its local invariants are left-invariant and real-analytic, too, including its Levi-Civita connection $\nabla$, its Ricci curvature $\Ric$, and the covariant derviatives $\nabla^k \Ric$ thereof. On the other hand, \cite[Proposition 3.5]{FeffermanGraham} gives that for an ambient metric in normal form \eqref{equation:normal-form} each of the derivatives $\partial_{\rho}^m g(u, \rho) \vert_{\rho = 0}$ of $g(u, \rho)$ is a certain linear combination of contractions of Ricci curvature and its covariant derviatives, and so in particular these derivatives are left-invariant and real-analytic. Thus, so is the real-analytic function $g(u, \rho)$ they define, and hence substituting this function in \eqref{equation:normal-form} yields a real-analytic ambient metric $\wtg$ with the specified invariance property. By left-invariance of $g_{ab}(h, \rho)$ and $(E_a)$, for all $h \in G$ and coframe elements $E_a, E_b$ we have $g_{ab}(h, \rho)(E_a\vert_h, E_b\vert_h) = g_{ab}(\id, \rho)(E_a\vert_{\id}, E_b\vert_{\id})$, so $g_{ab}(h, \rho)$ does not depend on $h$ and hence we may write it as a function $g_{ab}(\rho)$ of $\rho$ alone.
	\item The condition $\wt\Ric = 0$ is equivalent to the vanishing of $\wtR_{ab}$, $\wtR_{a\infty}$, and $\wtR_{\infty\infty}$, so its suffices to express each of those as differential expressions in $g_{ab}(\rho)$ involving only derivatives with respect to $\rho$. Note that the expression for $R_{\infty\infty}$ in \eqref{equation:ambient-metric-PDE} already has this form.

The only term in the expression for $\wtR_{ab}$ in \eqref{equation:ambient-metric-PDE} that depends on derivatives of $g$ in directions other than $\rho$ is $R_{ab}$, the Ricci curvature. Let $C_{ab}^c$ be the structure constants of $G$ with respect to the frame $(E_a)$, that is, the unique constants that satisfy $[E_a, E_b] = C_{ab}^c E_c$ for all $a, b, c$; by the Koszul formula, the Levi-Civita connection $\nabla$ of a left-invariant metric $g$ on $G$ is given by
\begin{equation}\label{equation:Levi-Civita-left-invariant}
	\nabla_{E_a} E_b = \tfrac{1}{2} \left( C_{ab}^c + g_{ad} C_{eb}^d g^{ec} - g_{bd} C_{ae}^d g^{ec} \right) E_c .
\end{equation}
We can use this formula to express the Riemannian curvature tensor $R_{abcd}$ and hence the Ricci curvature $R_{ab}$ as a sum of formal contractions of the metric $g_{ab}$, its inverse, $g^{ab}$, and the structure constants, $C_{ab}^c$ of $G$. (The formula is unwieldy so we do not reproduce it here.

By computing directly using \eqref{equation:Levi-Civita-left-invariant} again along with the formula for the Levi-Civita connection $\wt\nabla$ of an ambient metric in normal form \eqref{equation:normal-form} we can express the component $\wtR_{a \infty}$ in terms of $g_{ab}$, $g^{ab}$, $g_{ab}'$, and $C_{ab}^c$, as
\[
	\wtR_{a\infty} = \tfrac{1}{2} C_{ab}^d g'_{dc} g^{cb} + \tfrac{1}{2} C_{bc}^b g_{ad}' g^{dc} - \tfrac{1}{4} C_{cd}^b g'_{ef} g_{ab} g^{ce} g^{df} . \qedhere
\]
\end{enumerate}
\end{proof}

\begin{remark}
There is an even-dimensional analogues of Theorem \ref{theorem:Fefferman-Graham}, but it is more subtle: In short, for a conformal manifold of even dimension $n$, both formal existence and uniqueness of ambient metrics are guaranteed roughly only to order $\frac{n}{2}$ in $\rho$ \cite[\S3]{FeffermanGraham}. There is a corresponding even-dimensional version of Proposition \ref{proposition:left-invariant-ambient-metric} but we delay its precise statement to \cite{WillseLeftInvariant}.
\end{remark}

\begin{remark}
Any left-invariant bilinear form on $G$ is determined by its restriction to $T_{\id} G$, which we may identify with the Lie algebra $\mfg$ of $G$. Thus, we can regard the ambient metric system $\wt\Ric = 0$ as an ordinary differential equation on $\bigodot^2 \mfg^*$.
\end{remark}

\begin{remark}
Proposition \ref{proposition:left-invariant-ambient-metric} and the following remarks hold just as well if one replaces \textit{left-invariant} with \textit{right-invariant} in all instances.
\end{remark}

\section{A new explicit ambient metric}\label{section:new-explicit-ambient-metric}

With some effort, one can produce an explicit ambient metric for the conformal structure $(H, \mbc_{\mbE})$. Let $(e^a)$ denote the left-invariant coframe on $H$ dual to the frame $(E_a)$ defined in \S\ref{subsection:doubrov-example}. Then, $\mbc_{\mbE}$ has the representative
\begin{equation}\label{equation:representative-metric}
	g_{\mbE} = -2 (e^1)^2 - 2 e^1 e^3 + (e^2)^2 + 4 e^2 e^5 + 4 e^3 e^4 .
\end{equation}
\begin{proposition}
The metric
\begin{multline}\label{equation:explicit-ambient-metric}
	\wtg_{\mbE}
        = 2 \rho \,dt^2 + 2 t \,dt \,d\rho \\
            \qquad + t^2 \bigg(-\frac{4}{2 + \rho} (e^1)^2 - \frac{4 + 3 \rho}{2 + \rho} e^1 e^3 + (e^2)^2 + 2\sqrt{2 (2 + \rho)} e^2 e^5 \\- \frac{\rho^2}{16 (2 + \rho)} (e^3)^2 + 2 \sqrt{2 (2 + \rho)} e^3 e^4 \bigg)
\end{multline}
on $\wtH := \bbR_+ \times H \times (-2, \infty)$ is an ambient metric for the (real) left-invariant conformal structure $(H, \mbc_{\mbE})$ induced by the (real) distribution $\mbE$.
\end{proposition}
\begin{proof}
Evaluating the quantity in parentheses at $\rho = 0$ shows that $\wtg_{\mbE}$ is in normal form with respect to $g_{\mbE}$, and computing directly shows that $\wt\Ric := \Ric^{\wtg_{\mbE}} = 0$.
\end{proof}

Since solving the system \eqref{equation:ambient-metric-PDE} outright is typically difficult, even for left-invariant conformal structures, we indicate briefly a na\"{i}ve method for producing the explicit solution \eqref{equation:explicit-ambient-metric} for this particular case. First, note that for our chosen left-invariant frame $(E_a)$ most of the structure constants $C_{ab}^c$ (five of the $50$ with $a < b$) are zero, as are most of the coefficients of the representative metric $g_{\mbE}$ and Ricci curvature
\begin{equation}\label{equation:representative-metric-Ricci}
	\Ric^g(0) = \Ric^{g_{\mbE}} = (e^1)^2 - \tfrac{5}{4} e^1 e^3 + \tfrac{1}{4} (e^2)^2 + \tfrac{5}{2} e^2 e^5 + \tfrac{5}{2} e^3 e^4 .
\end{equation}

After computing the first several derivatives $\smash{g_{ab}^{(m)}(0)}$ of $g_{ab}(\rho)$ using the algorithm described after \eqref{equation:ambient-metric-PDE} some visible patterns emerge. First, only certain components $g_{ab}(\rho)$ have nonzero derivatives among them: Up to the symmetry $g_{ab} = g_{ba}$, only $g_{11}, g_{13}, g_{22}, g_{25}, g_{33}, g_{34}$ are nonzero.\footnote{For any ambient metric in normal form with respect to a metric $g$, the linear term in the Taylor expansion of $g(u, \rho)$ about $\rho = 0$ is a particular linear combination of $g$ and $\Ric^g$ \cite[(3.6)]{FeffermanGraham}, so it follows from \eqref{equation:representative-metric} and \eqref{equation:representative-metric-Ricci} without further computation that this claim is true for the first derivative, $g_{ab}'(0)$.} Second, the truncated Taylor series at $\rho = 0$ of $g_{25}, g_{34}$ agree. Third, the truncated Taylor series at $\rho = 0$ of the co\"{e}fficients $g_{11}(\rho), g_{13}(\rho), g_{22}(\rho), g_{33}(\rho)$ coincide with those of simple rational functions, namely, those in \eqref{equation:explicit-ambient-metric}.

These observations suggest the ansatz
\begin{multline*}
g(\rho) = -\frac{4}{2 + \rho} (e^1)^2 - \frac{4 + 3 \rho}{2 + \rho} e^1 e^3 + (e^2)^2 + 2 a(\rho) e^2 e^5 \\- \frac{\rho^2}{16 (2 + \rho)} (e^3)^2 + 2 a(\rho) e^3 e^4
\end{multline*}
for the quantity $g(h, \rho)$ in the normal form \eqref{equation:normal-form} for the ambient metric $\wtg$ of $\mbc_{\mbE}$. Then, decomposing the equation $\Ric^{\wtg} = 0$ with respect to the frame $(\partial_t, E_a, \partial_{\rho})$ yields a system of ordinary differential equations in $a(\rho)$. The full system, which has eight nonzero components, is too large to reproduce here, but, e.g., the component $\Ric^{\wtg}(E_2, \partial_{\rho}) = 0$ is a relatively simple first-order equation:
\[
	\frac{3 [2 (2 + \rho) a'(\rho) - a(\rho)]}{4 (2 + \rho)^2 a(\rho)} = 0 .
\]
This is separable and so can be solved readily by hand, and using the initial value $h(0) = g_{25}(0) = g_{34}(0) = 2$ gives the candidate solution $a(\rho) = \sqrt{2 (2 + \rho)}$. Substituting shows that this candidate is a solution to the system, so the resulting metric $\wtg_{\mbE}$ is indeed an ambient metric for $\mbc_{\mbE}$.

\begin{remark}\label{remark:nonpolynomial}
This is a first example of an explicit, closed-form ambient metric not polynomial in the ambient co\"{o}rdinate $\rho$. Since this article was first uploaded to the arXiv, more examples were produced in \cite{AndersonLeistnerNurowski}, including some of the ambient metrics $\wtg_{f, h}$ of conformal structures induced by $(2, 3, 5)$ distributions $\mbD_{f, h}$ given respectively in Monge normal form by the functions $F_{f, h}(x, y, p, q, z) := q^2 + f(x, y, p) + h(x, y) z$. As an anonymous referee observed, a priori the conformal structure $\mbc_{\mbE}$ may be (locally) equivalent to the conformal structure $\mbc_{f, h}$ induced by a distribution $\mbD_{f, h}$ for some $f, h$---and hence $\wtg_{\mbE}$ may occur (at least up to local isometry) in the family $\wtg_{f, h}$---but it turns out this is not the case. We give a proof of this assertion here, in part because the techniques used may be of independent interest.

Suppose there are such $f, h$. We first show briefly that local equivalence of $\mbc_{\mbE}$ and $\mbc_{f, h}$ implies local equivalence of $\mbE$ and $\mbD_{f, h}$: Theorem C of \cite{SagerschnigWillse} states that if for two different $(2, 3, 5)$ distributions $\mbD, \mbD'$ the induced conformal structures satisfy $\mbc_{\mbD} = \mbc_{\mbD'}$ then that conformal structure admits a so-called almost Einstein scale (see \S2.5 of that reference), in which case Proposition A of that reference implies that the normal conformal holonomy of that conformal structure is a proper subgroup of $\G_2^*$. On the other hand, Proposition \ref{proposition:G2-holonomy} below shows that the metric holonomy $\Hol(\wtg_{\mbE})$ is equal to $\G_2^*$, and by \cite[Corollary 1.2]{CGGH} it follows that the conformal holonomy $\Hol(\mbc_{\mbE})$ of $\mbc_{\mbE}$ is the full group $\G_2^*$, so $\mbc_{\mbE}$ does not admit an almost Einstein scale; hence $\mbE$ and $\mbD_{f, h}$ are locally equivalent.

Now, as observed in the proof of Proposition \ref{proposition:counterexample}, the fundamental curvature $A_{\mbE}$ of $\mbE$ has a quadruple root at all points, and hence so does the fundamental curvature $A_{f, h}$ of $\mbD_{f, h}$, which computing (using the algorithm in \cite[\S5]{GrahamWillse}) gives is
\[
	A_{f, g} = \left[\tfrac{1}{20} (3 f_{ppp} - 62 h_y) \,dq \,dx^3 + G(x, y, p, q, z) \,dx^4\right]\vert_{\mbD_{f, h}}
\]
for a particular function $G$. In particular, since $A_{f, h}$ has a quadruple root at all point, $f, h$ must together satisfy $3 f_{ppp} - 62 h_y = 0$, and differentiating both sides of this equation by $p$ gives that $f$ satisfies $f_{pppp} = 0$. The explicit formulae in \cite[(3.5), Theorem 3.2]{AndersonLeistnerNurowski} show that this latter condition implies that there is a representative metric $g_{f, h} \in \mbc_{f, h}$ for which the ambient metric in normal form with respect to $g_{f, h}$ is linear in $\rho$.

We can thus establish the claim by showing there is no $g \in \mbc_{\mbE}$ for which the ambient metric for $\mbc_{\mbE}$ in normal form with respect to $g$ is linear in $\rho$.

For a general conformal structure $\mbc$ and choice $g \in \mbc$ of representative metric consider the Taylor series expansion
\[
	g(u, \rho) \sim g(u) + \mu^{(1)}(u) \rho + \mu^{(2)}(u) \rho^2 + \cdots
\]
about $\rho = 0$ of the quantity $g(u, \rho)$ in \eqref{equation:normal-form}. The first equation in \cite[(3.18)]{FeffermanGraham} gives that (in dimension $5$) the co\"{e}fficient $\smash{\mu^{(2)}_{ab}}$ of the quadratic term is $-B_{ab} + P_a{}^c P_{bc}$, where (again specializing to dimension $5$) $P$ is the Schouten tensor of $g$, $P_{ab} = \frac{1}{3}\left(R_{ab} - \frac{1}{8} g^{cd} R_{cd} g_{ab}\right)$, $B$ is the Bach tensor of $g$, $B_{ab} = P_{b[c, a]}{}^c + P^{cd} W_{cabd}$, and $W$ is the Weyl tensor of $g$. Using the conformal invariance of $W$ and the conformal transformation laws for $B$ and $P$ \cite[pp. 56, 71, resp.]{FeffermanGraham} thus yields a formula for the co\"{e}fficient $\widehat{\mu}^{(2)}$ of the Taylor series expansion of the quantity $\wh g(u, \rho)$ in \eqref{equation:normal-form} for an ambient metric in normal form with respect to the general representative $e^{2 \Upsilon} g \in \mbc$ in terms of the arbitrary smooth function $\Upsilon$ and $g$ and its tensor invariants. (The formula is unwieldy and unenlightening, so we do not reproduce it here.) In the case of the metric $g_{\mbE} \in \mbc_{\mbE}$, computing gives, for example, that $\widehat{\mu}^{(2)}(E_1, E_1) = -\frac{1}{2} e^{-2 \Upsilon}$, which does not vanish for any $\Upsilon$. Thus, $\mbc_{\mbE}$ admits no representative metric with respect to which the normal-form ambient metric is linear in $\rho$, and hence it cannot be conformally equivalent to $\mbc_{f, g}$ for any $f, g$.
\end{remark}

\subsection{$\G_2^*$ holonomy}

Here we show that the metric $\wtg_{\mbE}$ produced in the previous subsection has holonomy equal to $\G_2^*$. We use the following theorem:

\begin{theorem}\cite[Theorem 1.1]{GrahamWillse}\label{theorem:graham-willse}
Let $\mbD$ be an oriented, real-analytic $(2, 3, 5)$ distribution. Then, the metric holonomy $\Hol(\wtg_{\mbD})$ of any real-analytic ambient metric $\wtg_{\mbD}$ for the conformal structure $\mbc_{\mbD}$ is contained in $\G_2^*$.
\end{theorem}

\begin{proposition}\label{proposition:G2-holonomy}
The metric holonomy $\Hol(\wtg_{\mbE})$ of the ambient metric $(\wtH, \wtg_{\mbE})$ is equal to $\G_2^*$.
\end{proposition}
\begin{proof}
By Theorem \ref{theorem:graham-willse}, $\Hol(\wtg_{\mbE}) \leq \G_2^*$, so to prove the claim, it suffices to show that the dimension of the holonomy group (or equivalently its Lie algebra) is equal to $\dim \G_2^* = 14$. (In fact, since the maximal subgroups of $\G_2^*$ are the maximal parabolic subgroups, all of which have dimension $9$, it suffices to show that $\dim \Hol(\wtg_{\mbE}) > 9$.)

By the Ambrose-Singer Theorem \cite{AmbroseSinger}, the Lie algebra of $\Hol(\wtg_{\mbE})$ contains (in fact, since $\wtg_{\mbE}$ is real-analytic, is equal to) the \textit{infinitesimal holonomy algebra} of $\wtg_{\mbE}$ at any point $\wt u \in \wt H$: This is the Lie algebra $\mfhol_{\wt u}(\wtg_{\mbE}) \subseteq \mfso((\wtg_{\mbE})_{\wt u})$ generated by the value of its curvature $\wtR$ and the derivatives thereof at $\wt u$, or more precisely, the endomorphisms $\wt{\nabla}_{Z_k} \cdots \wt{\nabla}_{Z_1} \wt{R}_{\wt u} (X, Y)$, where $X, Y, Z_1, \ldots, Z_k \in T_{\wt u} \wt H$ and $\wt{\nabla}$ is the Levi-Civita connection of $\wtg_{\mbE}$.
Computing gives that the image of $\wt{R}_{\wt u}$ in $\End(T_{\wt u} \wtH)$ (where the basepoint $\wt u \in \wtH$, which we suppress below, is any point with $t = 1, \rho = 0$), is spanned by the linearly independent endomorphisms $\wtR_{12}, \wtR_{14}, \wtR_{16}, \wtR_{23}, \wtR_{24}, \wtR_{26}$, where $\wtR_{ab} := \wtR(E_a, E_b)$ and $E_6 := \partial_{\rho}$. Those elements, together with $\wt{\nabla}_{E_1} \wtR_{12}$, $\wt{\nabla}_{E_2} \wtR_{12}$, $\wt{\nabla}_{E_3} \wtR_{12}$, $\wt{\nabla}_{E_4} \wtR_{12}$, $\wt{\nabla}_{\partial_{\rho}} \wt{R}_{12}$, $\wt{\nabla}_{E_2} \wtR_{14}$, $\wt{\nabla}_{E_2} \wtR_{16}$, $\wt{\nabla}_{\partial_{\rho}} \wtR_{16}$, comprise a basis of the space of endomorphisms generated by at most one derivative of $\wt{R}$. But this basis has $\dim \G_2^* = 14$ elements.
\end{proof}

Instead of using Theorem \ref{theorem:graham-willse} one can alternatively show the containment $\Hol(\wtg_{\mbE}) \leq \G_2^*$ by verifying that (1) the $3$-form
\begin{align}\label{equation:parallel-3-form}
	\wt\Phi :=& 
	- 6 \sqrt{2} \, t^2 \,dt \wedge e^{12}
	- \frac{2}{\sqrt{2 + \rho}} t^2 \, dt \wedge e^1 \wedge d\rho
	- \frac{8 + 3 \rho}{2 \sqrt{2}} t^2 \,dt \wedge e^{23} \\
	& \qquad
	- 4 \sqrt{2 + \rho} \, t^2 \, dt \wedge e^{24}
	+ \rho \sqrt{2 + \rho} \, t^2 dt \wedge e^{35} \nonumber \\
	& \qquad
	+ \frac{8 + 3 \rho}{4 \sqrt{2 + \rho}} t^2 \,dt \wedge e^3 \wedge d\rho
	+ 2 \sqrt{2} \, t^3 e^{125}
	- 2 \sqrt{2} \, t^3 e^{134} \nonumber \\
	& \qquad
	+ \frac{8 + 3 \rho}{2 \sqrt{2}} t^3 e^{235}
	- \frac{1}{2 \sqrt{2}} t^3 e^{23} \wedge d\rho
	+ \sqrt{2 + \rho} \, t^3 e^{35} \wedge d\rho \nonumber
\end{align}
is parallel with respect to $\wt\nabla$ and (2) for any (equivalently, every) $\wt u \in \wtH$, the stabilizer of $\wt\Phi_{\wt u}$ in $\GL(T_{\wt u} \wtH)$ is isomorphic to $\G_2^*$ and contained in $\SO((\wt{g}_{\mbE})_u)$; here, $e^{ab} := e^a \wedge e^b$ and $e^{abc} := e^a \wedge e^b \wedge e^c$.

\begin{remark}
The result \cite[Theorem 1.2]{GrahamWillse} gives conditions on real-analytic, oriented $(2, 3, 5)$ distributions $\mbD$ that are together sufficient to guarantee that any associated real-analytic ambient metric has holonomy $G_2^*$. Though the conditions hold for almost all such distributions (in a sense that can be made precise), they do not all hold for $\mbE$, and hence that result cannot be used to prove Proposition \ref{proposition:G2-holonomy}: One of the conditions is that there is a point $u$ at which the fundamental curvature $A_u$ does not have a multiple root in $\mbD_u \otimes \bbC$, but as we found in the proof of Proposition \ref{proposition:counterexample} for $\mbE$ the fundamental curvature has a quadruple root at every point.
\end{remark}

\subsection{Projective structure with normal projective holonomy $\G_2^*$}
\label{subsection:projective-structure}

By \cite{GoverPanaiWillse} we may regard the real-analytic ambient manifold of a real-analytic conformal structure of odd dimension $n \geq 3$ and signature $(p, q)$ as the so-called Thomas cone of a projective structure of dimension $n + 1$, which is defined on the space of orbits of the $\bbR_+$-action defined by the dilations $\delta_s$. This identifies the Levi-Civita connection of the ambient metric with the normal tractor connection---a natural, projectively invariant vector bundle connection on a particular vector bundle of rank $n + 2$ over and canonically associated to the projective manifold---whose holonomy is thus reduced to $\Ooperator(p + 1, q + 1)$.

Thus, the metric $\wtg_{\mbE}$ can be used to construct an explicit $6$-dimensional projective structure $(H \times (-2, \infty), \mbp)$ whose normal projective holonomy is equal to $\G_2^*$. Computing the connection forms of $\wt\nabla$ and applying the formula in \cite[Remark 6.6]{GoverPanaiWillse} (using the section $H \times (-2, \infty) \to \wtH \leftrightarrow \bbR_+ \times H \times (-2, \infty)$ with image $\{t = 1\}$) gives an explicit representative connection $\nabla$ in $\mbp$; the connection is characterized by the covariant derivative formulae in the appendix.

Example 7.1 of \cite{GoverPanaiWillse} also gives a $1$-parameter family of nontrivial $6$-dimensional projective structures with normal holonomy contained in $\G_2^*$, but for every member of that family the containment is proper.

\pagebreak

\appendix

\section{Data for the connection in Subsection \ref{subsection:projective-structure}}

Here we give the nonzero connection forms $\omega_{ab}^c$ of the connection $\nabla$ defined in \S\ref{subsection:projective-structure} in the frame $(E_1, \ldots, E_6)$, where we set $E_6 = \partial_{\rho}$, on $H \times (-2, \infty)$. These are characterized by $\nabla_{E_a} E_b = \omega_{ab}^c E_c$.
\begin{multline*}
	-\tfrac{1}{16} \omega_{11}^6 = \omega_{12}^4 = \omega_{21}^4
		= -\frac{1 + \rho}{2 (2 + \rho)^2}, \\
	-2 \omega_{12}^1 = -\omega_{13}^2 = \omega_{15}^4 = -2 \omega_{21}^1 = -\omega_{31}^2 = \omega_{51}^4
		= \frac{4 + 3 \rho}{2 \sqrt{2} (2 + \rho)^{3 / 2}}, \\
	\omega_{12}^3 = -\omega_{14}^5 = -\tfrac{8}{3} \omega_{16}^4 = \omega_{21}^3 = -8 \omega_{26}^5 = -\omega_{41}^5 = -\tfrac{8}{3} \omega_{61}^4 = -8 \omega_{62}^5
		= \frac{\sqrt{2}}{(2 + \rho)^{3 / 2}}, \\
	\omega_{13}^5 = \omega_{31}^5
		= \frac{4 + 3 \rho}{4 (2 + \rho)^2}, \qquad
	\omega_{13}^6 = \omega_{31}^6
		= \frac{8 + 8 \rho + 3 \rho^2}{2 (2 + \rho)^2}, \\
	\omega_{16}^1 = \omega_{22}^5 = -2 \omega_{26}^2 = -2 \omega_{36}^3 = -2 \omega_{46}^4 = -2 \omega_{56}^5 \qquad\qquad\qquad \\ \qquad\qquad\qquad = \omega_{61}^1 = -2 \omega_{62}^2 = -2 \omega_{63}^3 = -2 \omega_{64}^4 = -2 \omega_{65}^5
		= -\frac{1}{2 (2 + \rho)} , \\
	\omega_{22}^2 = 2 \omega_{24}^4 = -\omega_{25}^5 = 2 \omega_{42}^4 = -\omega_{52}^5
		= \frac{1}{\sqrt{2 (2 + \rho)}} , \\
	-\omega_{22}^6 = 2 \omega_{24}^1 = \omega_{34}^2 = -\omega_{35}^3 = -2 \omega_{42}^1 = \omega_{52}^2 = \omega_{53}^3 = -\omega_{54}^4 = -\omega_{55}^5
		= 1 , \\
	\omega_{23}^1 = \omega_{32}^1
		= -\frac{\rho (4 + 3 \rho)}{32 \sqrt{2} (2 + \rho)^{3 / 2}} , \\
	\omega_{23}^3 = \omega_{32}^2 = -\tfrac{8}{5} \omega_{36}^4 = -\tfrac{8}{5} \omega_{63}^4
		= \frac{\rho}{4 \sqrt{2} (2 + \rho)^{3 / 2}} , \qquad
	\omega_{23}^4 = \omega_{32}^4
		= -\frac{\rho (1 + \rho)}{16 (2 + \rho)^2} \\
	\omega_{25}^6 = \omega_{34}^6 = \omega_{43}^6 = \omega_{53}^6 = \omega_{52}^6
		= -\frac{4 + \rho}{\sqrt{2 (2 + \rho)}} , \\
	\omega_{33}^2 = - 2 \omega_{35}^4 = - 2 \omega_{53}^4
		= -\frac{\rho^2}{8 \sqrt{2} (2 + \rho)^{3 / 2}} , \qquad
	-2 \omega_{33}^5 = \omega_{33}^6
		= -\frac{\rho^2}{8 (2 + \rho)^2} , \\
	\omega_{34}^5 = \omega_{43}^5
		= -\frac{8 + 5 \rho}{4 \sqrt{2} (2 + \rho)^{3 / 2}} , \qquad
	\omega_{36}^1 = \omega_{63}^1 = -\frac{3 \rho}{32 (2 + \rho)} .
\end{multline*}

%

\end{document}